\documentclass[a4paper,12pt]{article}
\usepackage{amsfonts}
\usepackage{pifont}
\usepackage{bm}
\usepackage{latexsym,amsmath,amssymb,cite,amsthm}
\numberwithin{equation}{section}
\theoremstyle{plain}
\newtheorem{theorem}{Theorem}[section]

\newtheorem{lemma}[theorem]{Lemma}
\newtheorem{proposition}[theorem]{Proposition}
\theoremstyle{definition}
\newtheorem{definition}[theorem]{Definition}
\theoremstyle{remark}

\newcommand{\Ric}{{\rm{Ricci}}}

\newcommand{\bRic}{{\bf Ricci}}
\newcommand{\bRicn}{{\bf Ricci}_N}
\newcommand{\tr}{{\rm{tr}}}
\renewcommand{\H}{{\mathrm{H}}}

\newcommand{\mm}{\mathfrak m}
\newcommand{\ms}{(X,\d,\mm)}

\newcommand{\rcdkn}{{\rm RCD}^*(K, N)}
\newcommand{\rcd}{{\rm RCD}(K, \infty)}

\newcommand{\N}{\mathbb{N}}

\newcommand{\R}{\mathbb{R}}

\newcommand{\Id}{{\rm Id}}
\renewcommand{\d}{{\mathrm d}}

\newcommand{\D}{{\mathrm D}}
\newcommand{\restr}[1]{\lower3pt\hbox{$|_{#1}$}}
\newcommand{\la}{{\langle}}
\newcommand{\ra}{{\rangle}}

\newcommand{\nchi}{{\raise.3ex\hbox{$\chi$}}}

\title{{\bf  Ricci tensor on $\rcdkn$ spaces}
}

\begin{document}
\author{ Bang-Xian Han \thanks
{ Hausdorff Center for Mathematics, University of Bonn,
han@iam.uni-bonn.de}
}

\date{\today}
\maketitle

\begin{abstract}
We obtain an improved Bochner inequality based on the curvature-dimension condition $\rcdkn$ and  propose a definition of $N$-dimensional Ricci tensor on metric measure spaces.
\end{abstract}

\textbf{Keywords}: curvature-dimension condition, Bakry-\'Emery theory, Bochner inequality, Ricci tensor, metric measure space.\\

\tableofcontents


\section{Introduction}
Let $M$ be a Riemannian manifold equipped with a metric tensor $\la \cdot, \cdot \ra: [TM]^2 \mapsto C^\infty(M)$. We have the Bochner formula
\begin{equation}\label{bf}
\Gamma_2(f)=\Ric(\nabla f, \nabla f)+| \H_f |_{\rm HS}^2,
\end{equation}
valid for any smooth function $f$, where $| \H_f |_{\rm HS}$ is the Hilbert-Schmidt norm of the Hessian $\H_f:=\nabla \d f$ and the operator $\Gamma_2$ is defined by
\[
\Gamma_2(f):=\frac12 L \Gamma(f,f) -\Gamma(f, Lf),\qquad\Gamma(f,f):=\frac12L(f^2)-fLf
\]
where $\Gamma(\cdot, \cdot)=\la \nabla \cdot, \nabla \cdot \ra$, and $L=\Delta$ is the Laplace-Beltrami operator.

In particular, if the Ricci curvature of $M$ is bounded from below by $K$, i.e. $\Ric(v,v)(x)\geq K |v|^2(x)$ for any $x\in M$ and $v\in T_xM$,  and the dimension is bounded from above by $N\in [1,\infty]$, we have the Bochner inequality
\begin{equation} \label{bekn}
\Gamma_2(f) \geq \frac1N(L f)^2+K\Gamma(f).
\end{equation}
Conversely, it is not hard to show that the validity of \eqref{bekn} for any smooth function $f$ implies that the manifold has lower Ricci curvature bound  $K$ and upper dimension bound  $N$, or in short that it is a ${\rm CD}(K,N)$ manifold.

Being this characterization of the ${\rm CD}(K,N)$ condition only based on properties of $L$, one can take \eqref{bekn} as a definition of the ${\rm CD}(K,N)$ condition with respect to the diffusion operator $L$. This was the approach suggested by Bakry-\'Emery in \cite{BE-D}, we refer to \cite{BGL-A} for an overview on this subject.

Following this line of thought, one can wonder  whether in this framework we can recover the definition of the Ricci curvature tensor and deduce from \eqref{bekn} that it is bounded from below by $K$. From \eqref{bf} we see that a natural definition is
\begin{equation}
\label{eq:defric}
\Ric(\nabla f,\nabla f):=\Gamma_2(f)-| \H_f |_{\rm HS}^2,
\end{equation}
and it is clear that if $\Ric\geq K$, then \eqref{bekn} holds with $N=\infty$. There are few things that need to be understood in order to make definition \eqref{eq:defric} rigorous and complete in the setting of diffusion operators:
\begin{itemize}
\item[1)] If our only data is the diffusion operator $L$, how can we give a meaning to the Hessian term in \eqref{eq:defric}?
\item[2)] Can we  deduce that the Ricci curvature defined as in \eqref{eq:defric} is actually bounded from below by $K$ from the assumption \eqref{bekn}?
\item[3)] Can we include the upper bound on the dimension in the discussion? How the presence of $N$ affects the definition of the Ricci curvature?
\end{itemize}
This last question has a well known answer thanks to the work of Bakry-\'Emery: it turns out that the correct thing to do is to define, for every $N\geq 1$, a sort of `$N$-dimensional' Ricci tensor as follows:
\begin{equation}
\label{eq:defricn}
\Ric_N(\nabla f, \nabla f):=\left\{\begin{array}{ll}
{ \Gamma}_2(f)-| \H_f |_{\rm HS}^2-\frac1{N-n}(\tr \H_f -L f)^2,&\qquad \text{if }N>n,\\
{ \Gamma}_2(f)-| \H_f |_{\rm HS}^2-\infty(\tr \H_f -L f)^2,&\qquad \text{if }N=n,\\
-\infty,&\qquad\text{if }N<n,
\end{array}\right.
\end{equation}
where $n$ is the dimension of the manifold (recall that on a weighted manifold in general we have ${\rm tr}\H_f\neq \Delta f$). It is then not hard to see that if $\Ric_N\geq K$ then indeed \eqref{bekn} holds.

It is harder to understand how to  go back and prove that $\Ric_N\geq K$ starting from \eqref{bekn}. A first step in this direction, which answers (1), is to notice that in the smooth setting the  identity
\[
2\H_f(\nabla g,\nabla h)=\Gamma(g,\Gamma(f,h))+\Gamma(h,\Gamma(f,g))-\Gamma(f,\Gamma(g,h))
\]
for any smooth $g,h$ totally characterizes the Hessian of $f$, so that the same identity can be used to define the Hessian starting from a diffusion operator only. The question is then whether one can prove any efficient bound on it starting from \eqref{bekn} only. The first result in this direction was obtained by Bakry in \cite{B-T} and \cite{B-H}, and  recently Sturm \cite{S-R} concluded the argument showing that \eqref{bekn} implies $\Ric_N\geq K$. In Sturm's approach, the operator $\Ric_N$ is not defined as in \eqref{eq:defricn}, but rather as
\begin{equation}
\label{eq:sturm}
\Ric_N(\nabla f,\nabla f)(x):=\inf_{g\ : \ \Gamma(f-g)(x)=0}  \Gamma_2(g)(x)-\frac{(L g)^2(x)}N
\end{equation}
and it is part of his contribution in the proof that this definition is equivalent to \eqref{eq:defricn}.

\bigskip

All this for smooth, albeit possibly abstract, structures. On the other hand, there is as of now a quite well established theory of (non-smooth) metric measure spaces satisfying a curvature-dimension condition: that of $\rcdkn$ spaces introduced by Ambrosio-Gigli-Savar\'e (see \cite{AGS-M} and \cite{G-O}) as a refinement of the original intuitions of Lott-Sturm-Villani (see \cite{Lott-Villani09} and \cite{S-O1, S-O2}) and Bacher-Sturm (see \cite{BS-L}). In this setting, there is a very natural Laplacian and inequality \eqref{bekn} is known to be valid in the appropriate weak sense (see \cite{AGS-M}, \cite{AMS-N} and \cite{EKS-O}) and one can therefore wonder if even in this low-regularity situation one can produce an effective notion of $N$-Ricci curvature. Part of the problem here is the a priori lack of vocabulary, so that for instance it is unclear what a vector field should be.

In the recent paper \cite{G-N}, Gigli builds a differential structure on metric measure spaces suitable to handle the objects we are discussing  (see the preliminary section for some details). One of his results is to give a meaning to formula \eqref{eq:defric} on $\rcd$ spaces and to prove that the resulting Ricci curvature tensor, now measure-valued, is bounded from below by $K$. Although giving comparable results, we remark that the definitions used in \cite{G-N} are different from those in \cite{S-R}: it is indeed  unclear how to give a meaning to formula \eqref{eq:sturm} in the non-smooth setting, so that in \cite{G-N} the definition \eqref{eq:defric} has been adopted.

Gigli worked solely in the $\rcd$ setting. The contribution of the current work is to adapt Gigli's tool and Sturm's computations to give a complete description of the $N$-Ricci curvature tensor on $\rcdkn$ spaces for $N<\infty$.

In our main result Theorem \ref{th-ricci}, we prove that the $N$-Ricci curvature is  bounded from below by $K$ on a ${\rm RCD}(K',\infty)$ space if and only if the space is $\rcdkn$. As an application of our theory, it is possible study the curvature dimension condition of metric measure space under several transformations, see \cite{H-C} for more details.


\section{Preliminaries}
Let $M=\ms$ be a $\rcd$ metric measure space for some $K \in \R$ (or for simplicity, a ${\rm RCD}$ space). We denote the space of finite  Borel measures on $X$ by ${\rm Meas}(M)$, and equip it with the total variation norm $\| \cdot \|_{\rm TV}$.

The Sobolev space $W^{1,2}(M)$ is defined as in \cite{AGS-C}, and the minimal weak upper gradient (or call it weak gradient for simplicity) of a function $f \in W^{1,2}(M)$ is denoted by $|\D f|$. It is part of the definition of ${\rm RCD}(K,\infty)$ space that $W^{1,2}(M)$ is a Hilbert space, in which case $\ms$ is called infinitesimally Hilbertian space (see \cite{G-O} for more discussions).
In order to introduce the concepts of `tangent/cotangent vector field' in non-smooth setting, we will use the vocabulary  of  $L^\infty$-modules.

\begin{definition}[$L^2$-normed $L^\infty$-module]\label{L2norm}

Let $M=\ms$ be a metric measure space. A $L^2$-normed $L^\infty(M)$ module is a Banach space $({\bf B}, \| \cdot \|_{\bf B})$ equipped with a bilinear map
\begin{eqnarray*}
L^\infty(M) \times {\bf B} &\mapsto&  {\bf B},\\
 (f,v) &\mapsto& f \cdot v
\end{eqnarray*}
such that
\begin{eqnarray*}
(fg) \cdot v &=& f \cdot  (g \cdot v),\\
 {\bf 1} \cdot v &=& v
\end{eqnarray*}
for every $v \in {\bf B}$ and $f, g \in L^\infty(M)$, where ${\bf 1} \in L^\infty(M)$ is the function identically equal to 1 on $X$, and a `pointwise norm' $|\cdot|: {\bf B} \mapsto  L^2(M)$ which maps $v \in {\bf B}$ to  a non-negative function in $L^2(M)$
such that
\begin{eqnarray*}
\|v\|_{\bf B} &=&  \| |v| \|_{L^2}\\
|f\cdot v| &=& |f||v|,~~ \mm-\text {a.e}.
\end{eqnarray*}
for every $f \in L^\infty(M)$  and $v \in  {\bf B}$.



\end{definition}

Now we define the tangent and cotangent modules of $M$ which are particular examples of $L^2$-normed module.  We define the `Pre-Cotangent Module' $\mathcal{PCM}$ as the set consisting the elements of the from  $\{ (A_i, f_i )\}_{i  \in \mathbb{N}}$, where $\{A_i\}_{i \in \mathbb{N}}$ is a Borel partition of $X$, and $\{f_i\}_i$ are Sobolev functions such that $\sum_i \int_{A_i} |\D f_i|^2 <\infty$.

We define an equivalence relation on $\mathcal{PCM}$ via
\[
 \{ (A_i, f_i )\}_{i  \in \mathbb{N}} \thicksim \{ (B_j, g_j )\}_{j  \in \mathbb{N}}~~~\text {if}~~~~
 |\D(g_j-f_i)|=0, ~~\mm-\text {a.e. on}~ A_i \cap B_j.
\]
We denote the equivalence class of $\{ (A_i, f_i )\}_{i  \in \mathbb{N}}$ by $[ (A_i, f_i )]$. In particular, we call $[(X, f)]$ the differential of a Sobolev function $f$ and denote it by $\d f$.

Then we define the following operations:
\begin{itemize}\label{premodule}
\item [1)] $[ (A_i, f_i )]+[ (B_i, g_i )]:=[ (A_i \cap B_j, f_i+g_j )]$,\\

\item [2)] Multiplication by scalars: $\lambda [ (A_i, f_i )]:= [ (A_i, \lambda f_i )]$,\\

\item [3)] Multiplication by simple functions: $(\sum_j \lambda_j \nchi_{B_j})  [ (A_i, f_i )]:=[ (A_i \cap B_j, \lambda_j f_i )]$, \\
\item [4)] Pointwise norm: $|[ (A_i, f_i )]|:=\sum_i \nchi_{A_i} |\D f_i|$,
\end{itemize}
where $\nchi_A$ denote the characteristic function on the set $A$.

It can be seen that all the operations above are continuous on $\mathcal{PCM}/ \thicksim$ with respect to the norm $\| [ (A_i, f_i )] \|:= \sqrt{\int|[ (A_i, f_i )]|^2\,  \mm}$ and the $L^\infty(M)$-norm on the space of simple functions. Therefore we can extend them to the completion of $(\mathcal{PCM}/\thicksim, \| \cdot \|)$ and we denote this completion by $L^2(T^*M)$. As a consequence of our definition, we can see that $L^2(T^* M)$ is the $\| \cdot \|$ closure of $\{\sum_{i \in I} a_i \d f_i: |I|<\infty,  a_i \in L^\infty(M), f_i \in W^{1,2} \}$ (see Proposition 2.2.5 in \cite{G-N} for a proof).  It can also be seen from the definition and the infinitesimal Hilbertianity assumption on $M$ that $L^2(T^*M)$ is a Hilbert space equipped with the inner product induced by $\| \cdot \|$. Moreover, $(L^2(T^*M), \| \cdot \|, |\cdot |)$ is a $L^2$-normed module according to the Definition \ref{L2norm}, which we shall call cotangent module of $M$.

We then define the tangent module $L^2(TM)$ as  $\mathrm{Hom}_{L^\infty(M)}(L^2(T^*M), L^1(M))$, i.e. $T \in L^2(T^* M)$ if it is a continuous linear map from $L^2(T^*M)$  to $L^1(M)$ viewed as Banach spaces satisfying the homogeneity:
\[
T(f v)=f  T(v), ~~\forall v \in L^2(T^*M),~~ f \in L^\infty(M).
\]

It can be seen that $L^2(TM)$ has a natural $L^2$-normed $L^\infty(M)$-module structure and  is isometric to $L^2(T^*M)$ both as a module and as a Hilbert space. We denote the corresponding element of $\d f$ in $L^2(TM)$ by $\nabla f$ and call it the gradient of $f$ (see also the Riesz theorem for Hilbert modules in Chapter 1 of \cite{G-N}). The natural pointwise norm on $L^2(TM)$ (we also denote it by $| \cdot |$) satisfies $|\nabla f|=|\d f|=|\D f|$. We can also prove that  $\{\sum_{i \in I} a_i  \nabla f_i: |I|<\infty,  a_i \in L^\infty(M), f_i \in W^{1,2} \}$ is dense in $L^2(T M)$.

In other words,  we have a pointwise inner product $\la \cdot, \cdot \ra: [L^2(T^*M)]^2 \mapsto L^1(M)$ satisfying
\[
\la \d f, \d g \ra:= \frac14 \Big{(}|\D (f+g)|^2-|\D (f-g)|^2\Big{)}
\]
for $f, g \in W^{1,2}(M)$. Thus for any $g\in W^{1,2}(M)$, we can define the gradient $\nabla g$  as the element in $L^2(TM)$ such that $\nabla g (\d f)=\la \d f, \d g \ra, \mm$-a.e. for every $f \in W^{1,2}(M)$. Therefore, $L^2(TM)$ inherits a pointwise inner product from $L^2(T^*M)$ and we still use $\la \cdot, \cdot \ra$ to denote it.

Then we can define the Laplacian by duality (integration by part) in the same way as on a Riemannian manifold.

\begin{definition}
[Measure valued Laplacian, \cite{G-O, G-N}]
The space ${\rm D}({\bf \Delta}) \subset  W ^{1,2}(M)$ is the space of $f \in  W ^{1,2} (M)$ such that there is a measure ${\bf \mu}$ satisfying
\[
\int h \,\d {\bf \mu}= -\int \la \nabla h, \nabla f \ra \,   \d \mm, \forall h: M \mapsto  \R, ~~ \text{Lipschitz with bounded support}.
\]
In this case the measure $\mu$ is unique and we shall denote it by ${\bf \Delta} f$. If ${\bf \Delta} f \ll m$, we denote its density by $\Delta f$.
\end{definition}

In \cite{S-S} the author introduces the space of test functions ${\rm TestF}(M) \subset W^{1,2}(M)$,  which is defined as 
\[
{\rm TestF}(M):= \{f \in {\rm D} ({\bf \Delta}) \cap L^\infty: |\D f|\in L^\infty~~ {\rm and}~~~ \Delta f  \in W^{1,2}(M) \}.
\]
It is known from \cite{AGS-M} that ${\rm TestF}(M)$ is dense in $W^{1,2}(M)$ provided $M$ is $\rcd$. It is also proved in \cite{S-S} that ${\rm TestF}(M)$ is an algebra under the same assumption.
In particular, we know that $\la \nabla f, \nabla g \ra \in {\rm D}({\bf \Delta})\subset W^{1,2}(M)$ for any $f, g \in {\rm TestF}(M)$. Therefore we can define the Hessian and $\Gamma_2$ operator as follows.

Let $f \in {\rm TestF}(M)$. We define the Hessian $\H_f: \{ \nabla g: g \in {\rm TestF}(M)\}^2 \mapsto L^0(M)$ by
\[
2\H_f(\nabla g,\nabla h)=\la \nabla g, \nabla \la \nabla f, \nabla h \ra \ra +\la \nabla h, \nabla \la \nabla f, \nabla g \ra \ra-\la \nabla f, \nabla \la \nabla g, \nabla h \ra \ra
\]
 for any $g, h \in {\rm TestF}(M)$. Using the estimate obtained in \cite{S-S}, it can be seen that $\H_f$ can be extended to a symmetric $L^\infty(M)$-bilinear map on $L^2(TM)$ and continuous with values in $L^0(M)$, see Theorem 3.3.8 in \cite{G-N} for a proof.

Let $f,g \in {\rm TestF}(M)$. We define the measure ${\bf \Gamma}_2(f,g)$ as
\[
{\bf \Gamma}_2(f,g)=\frac12 {\bf \Delta} \la \nabla f, \nabla g \ra -\frac12 \big{(}\la \nabla f, \nabla \Delta g \ra+\la \nabla g, \nabla \Delta f \ra\big{)}\, \mm,
\]
and we put ${\bf \Gamma}_2(f):={\bf \Gamma}_2(f,f)$.

Then we recall some results on the non-smooth Bakry-\'Emery theory.

\begin{proposition} [Bakry-\'Emery condition, \cite{AGS-B}, \cite{EKS-O}]\label{becondition}
Let $M=\ms$ be  a ${\rm RCD}$ space. Then it is a $\rcdkn$ space with $K \in \R$ and $N \in [1, \infty]$ if and only if
\[
{\bf \Gamma}_2(f) \geq \Big {(} K |\D f|^2+ \frac1N (\Delta f)^2 \Big{)}\,\mm
\]
for any $f \in {\rm TestF}(M)$.
\end{proposition}

\begin{lemma}[\cite{S-S}]\label{self improve}
Let $M=\ms$ be a $\rcdkn$  space, $n \in \N$, $f_1, ... , f_n \in {\rm TestF}(M)$ and $\Phi \in C^\infty(\R^n)$
be with $\Phi(0)=0$. Put $ {\bf f}=(f_1, ... , f_n)$,
then $\Phi({\bf f}) \in {\rm TestF}(M)$.  In particular, ${\bf \Gamma}_2( \Phi({\bf f})) \geq \Big{(} K|\D \Phi({\bf f})|^2+\frac{(\Delta \Phi({\bf f}))^2}{N} \Big{)}\, \mm$.
\end{lemma}

\begin{lemma}[Chain rules, \cite{B-H}, \cite{S-S}]\label{chain}
Let $f_1, ... , f_n \in {\rm TestF}(M)$ and $\Phi \in C^\infty(\R^n)$ be with $\Phi(0)=0$. Put $ {\bf f}=(f_1, ... , f_n)$, then
\[
|\D \Phi({\bf f})|^2\,\mm=\mathop{\sum}_{i,j=1}^n \Phi_i\Phi_j ({\bf f}) \la \nabla f_i, \nabla f_j \ra\,\mm,
\]
\begin{eqnarray*}
{\bf \Gamma}_2(\Phi({\bf f}))&=& \mathop{\sum}_{i,j} \Phi_i\Phi_j ({\bf f}){\bf \Gamma}_2(f_i, f_j)\\ &+& 2\mathop{\sum}_{i,j,k} \Phi_i\Phi_{j,k} ({\bf f})\H_{f_i}(\nabla f_j, \nabla f_k)\,\mm\\
&+& \mathop{\sum}_{i,j,k,l} \Phi_{i,j}\Phi_{k,l} ({\bf f})\la \nabla f_i, \nabla f_k \ra \la \nabla f_j, \nabla f_l \ra\,\mm,
\end{eqnarray*}
and
\[
{\bf \Delta} \Phi({\bf f})=\mathop{\sum}_{i=1}^n \Phi_i({\bf f}){\bf \Delta} f_i+\mathop{\sum}_{i,j=1}^n \Phi_{i,j}({\bf f}) \la \nabla f_i, \nabla f_j \ra \,\mm.
\]

\end{lemma}

At the end of this section, we discuss  the  dimension of $M$ which is understood as the dimension of  $L^2(TM)$  as a $L^\infty$-module.
Let $A$ be a Borel set. We denote the subset of $L^2(TM)$ consisting of those $v$ such that $\nchi_{A^c} v=0$ by  $L^2(TM)\restr{A}$.

\begin{definition}[Local independence] Let $A$ be a Borel set with positive measure. We say that $\{v_i\}_{i=1}^n \subset L^2(TM)$ is independent on $A $ if
\[
\sum_i f_i v_i = 0, ~~\mm-\text {a.e.} ~~\text{on}~~ A
\]
holds if and only if $f_i = 0$  $\mm$-a.e. on $A$ for each $i$.

\end{definition}

\begin{definition}[Local span and generators] Let $A$ be a Borel set in $X$ and $V:=\{v_i\}_{i \in I} \subset L^2(TM)$. The span of $V$ on $A$, denoted by ${\rm Span}_A(V)$, is the subset of $L^2(TM)\restr{A}$  with the following property: there exist a Borel decomposition $\{A_n\}_{n\in \mathbb{N}}$ of $A$ and families of vectors $\{v_{i,n}\}_{i=1}^{m_n} \subset L^2(TM)$ and functions $\{f_{i,n}\}_{i=1}^{m_n} \subset L^\infty(M)$, $n=1,2,...,$ such that
\[
\nchi_{A_n}  v=\mathop{\sum}_{i=1}^{m_n} f_{i,n} v_{i,n}
\]
for each $n$.
We call the closure of ${\rm Span}_A(V)$  the space generated by $V$ on $A$.
\end{definition}

We say that $L^2(TM)$ is finitely generated if there is a finite family $v_1,...,v_n$ spanning $L^2(TM)$ on $X$, and locally finitely generated if there is a partition $\{E_i\}$ of $X$ such that $L^2(TM)\restr{E_i}$ is finitely generated for every $i \in \mathbb{N}$.

\begin{definition}[Local basis and dimension] We say that a finite set $v_1,...,v_n$ is a basis on a Borel set $A$ if it is independent on $A$ and ${\rm Span}_A\{v_1,...,v_n\} = L^2(TM)\restr{A}$.
If $L^2(TM)$ has a basis of cardinality $n$ on $A$, we say that it has dimension $n$ on A, or that its local dimension on $A$ is $n$.
If $L^2(TM)$ does not admit any local basis of finite cardinality on  any subset of $A$ with positive measure, we say that $L^2(TM)$ has infinite dimension on $A$.
\end{definition}

It can be proved  (see Proposition 1.4.4 in \cite{G-N} for example) that the definition of basis and dimension are well posed.
As a consequence of  this definition, we can prove the existence of a unique decomposition $\{ E_n\}_{n \in \mathbb{N} \cup \{\infty\}}$ of $X$ such that for each $E_n$ with positive measure, $n \in \mathbb{N} \cup \{\infty\}$, $L^2(TM)$ has dimension $n$ on $E_n$. Furthermore, thanks to the infinitesimal Hilbertianity we have  the following proposition.

\begin{proposition}[Theorem 1.4.11, \cite{G-N}]\label{decomposition}
Let $\ms$ be a $\rcd$ metric measure space. Then there exists  a unique decomposition $\{ E_n\}_{n \in \mathbb{N} \cup \{\infty\}}$ of $X$ such that
\begin{itemize}
\item For any $n \in \mathbb{N}$ and any $B \subset E_n$ with finite positive measure,  $L^2(TM)$ has a unit orthogonal basis $\{e_{i,n}\}_{i=1}^n$ on $B$,
\item For every subset $B$ of $E_\infty$ with finite positive measure, there exists a unit orthogonal set $\{e_{i,B}\}_{i \in \mathbb{N} \cup \{\infty\}} \subset L^2(TM)\restr{B}$  which generates $L^2(TM)\restr{B}$,
\end{itemize}
where unit orthogonal of a countable set $\{v_i\}_i\subset L^2(TM)$ on $B$ means $\la v_i, v_j \ra=\delta_{ij}$ $\mm$-a.e. on $B$.
\end{proposition}

\begin{definition}[Analytic Dimension]
We say that the dimension of $L^2(TM)$ is $k$ if $k=\sup \{n: \mm(E_n)>0\}$ where $\{ E_n\}_{n \in \mathbb{N} \cup \{\infty\}}$ is the decomposition given in Proposition \ref{decomposition}. We define the analytic dimension of $M$ as the dimension of $L^2(TM)$ and denote it by $\dim_{\rm max} M$.
\end{definition}


\section{Improved Bochner inequality}
In this part, we will study the dimension of $\rcdkn$ metric measure spaces and prove an improved Bochner inequality.

First of all, we have a lemma.

\begin{lemma}[Lemma 3.3.6, \cite{G-N}]\label{lemma-qf}
Let $\mu_i=\rho_i \,\mm+\mu_i^s, i=1,2,3$  be measures with $\mu_i^s \perp \mm$.  We assume that
\[
\lambda^2 \mu_1+2\lambda \mu_2 + \mu_3 \geq 0, ~~~~~\forall \lambda \in \R.
\]
Then we have
\[
\mu_1^s \geq 0,~~~\mu_3^s \geq 0
\]
and
\[
|\rho_2|^2 \leq \rho_1 \rho_3,~~~\mm-\text {a.e.}.
\]
\end{lemma}

Now we prove that $N$ is an upper bound of the dimensions of $\rcdkn$ spaces.

\begin{proposition}\label{finite dim}
Let $M=\ms$ be a $\rcdkn$ metric measure space. Then ${\dim}_{\rm max} M \leq N$. Furthermore, if the local dimension on a Borel set $E$ is $N$, we have
$\tr \H_f(x)=\Delta f (x)$ $\mm$-a.e. $x \in E$ for every $f \in {\rm TestF}$.

\end{proposition}
\begin{proof}
Let  $\{E_m\}_{m \in \mathbb{N} \cup \{\infty\}}$ be the partition of $X$ given by Proposition \ref{decomposition}. To prove ${\dim}_{\rm max} M \leq N$, it is sufficient to prove that for any $E_m$ with positive measure, we have $m \leq N$.

Then, let $m \in \mathbb{N} \cup \{\infty\}$ be such that $\mm(E_m)>0$, and $n \leq m$ a finite number. We define the function $\Phi(x, y, z_1, ... ,z_{n}):=\lambda (xy+ x)-by + \sum_i^{n} (z_i-c_i)^2-\sum_i^n c_i^2$ where $ \lambda,  b, c_i \in \R$. Then  we have
\begin{eqnarray*}
\Phi_{x,i}&=&0,~~~~~~~~~~~~~\Phi_{y,i}= 0, ~~~~~~~~~~~~~~~~\Phi_{i,j}= 2 \delta_{ij},~~~~~~~~~~~~~\Phi_{x,y}=\lambda\\
\Phi_x &=& \lambda y+\lambda,~~~~~~~~\Phi_y = \lambda x-b,~~~~~~~~~~\Phi_i =2(z_i-c_i).
\end{eqnarray*}

From Lemma \ref{self improve} we know
\[
{\bf \Gamma}_2( \Phi({\bf f})) \geq \left ( K|\D \Phi({\bf f})|^2 + \frac{(\Delta \Phi({\bf f}))^2 }{N} \right )\, \mm
\]
for any ${\bf f}= (f, g, h_1,...,h_n)$ where $f, g, h_1, ..., h_n \in {\rm TestF}$.

Combining the chain rules (see Lemma \ref{chain}), the inequality above becomes:
\begin{equation}\label{eqn-1}
{\bf A}(\lambda, b, {\bf c}) \geq \Big {(} KB(\lambda, b, {\bf c})+\frac1N C^2(\lambda, b, {\bf c}) \Big{)}\,\mm,
\end{equation}
where
\begin{eqnarray*}
{\bf A}(\lambda, b, {\bf c})&=& (\lambda f-b)^2 {\bf \Gamma}_2(g)+(\lambda g+\lambda)^2 {\bf \Gamma}_2(f)+\mathop{\sum}_{i,j} 4(h_i-c_i)(h_j-c_j){\bf \Gamma}_2(h_i,h_j)\\
&+& 2\lambda (g+1)(\lambda f-b){\bf \Gamma}_2(f,g)+\mathop{\sum}_i 4(\lambda g+\lambda)(h_i-c_i){\bf \Gamma}_2(f,h_i)\\ &+& \mathop{\sum}_i 4(\lambda f-b)(h_i-c_i){\bf \Gamma}_2(g,h_i) + 8\lambda \mathop{\sum}_i(h_i-c_i)\H_{h_i}(\nabla f, \nabla g)\,\mm \\ &+& 4\mathop{\sum}_i(\lambda g+\lambda) \H_f(\nabla h_i, \nabla h_i)\,\mm+4\mathop{\sum}_i(\lambda f-b) \H_g(\nabla h_i, \nabla h_i)\,\mm \\
&+&  4\lambda(\lambda f-b)\H_g(\nabla f, \nabla g)\,\mm+4\lambda (\lambda g+\lambda)\H_f(\nabla f,\nabla g)\,\mm\\ &+&8\mathop{\sum}_{i,j}(h_i-c_i)\H_{h_i}(\nabla h_j, \nabla h_j)\,\mm
+2 \lambda^2|\D f|^2 |\D g|^2\,\mm+2\lambda^2 |\la \nabla f, \nabla g \ra|^2\,\mm \\ &+& 4\mathop{\sum}_{i,j} |\la \nabla h_i, \nabla h_j \ra|^2\,\mm+
8\lambda \mathop{\sum}_i \la \nabla f, \nabla h_i \ra \la \nabla g, \nabla h_i \ra\,\mm
\\
B(\lambda, b, {\bf c})&=& (\lambda f-b)^2 |\D g|^2 +(\lambda g+\lambda)^2 |\D f|^2\\
&+&2(\lambda g+\lambda)(\lambda f-b) \la \nabla f, \nabla g \ra+4\mathop{\sum}_i (\lambda g+\lambda)(h_i-c_i) \la \nabla f, \nabla h_i\ra \\ &+&4\mathop{\sum}_i(\lambda f-b)(h_i-c_i)\la \nabla g, \nabla h_i\ra + 4\mathop{\sum}_{i,j}(h_i-c_i)(h_j-c_j) \la \nabla h_i, \nabla h_j \ra\\
C(\lambda, b, {\bf c})&=& (\lambda g+\lambda) \Delta f+(\lambda f-b) \Delta g+2\mathop{\sum}_i (h_i-c_i)\Delta h_i \\ &+& 2\lambda \la \nabla f, \nabla g \ra+2\mathop{\sum}_i |\D h_i|^2.
\end{eqnarray*}

Let $B$ be an arbitrary Borel set. From the inequality \eqref{eqn-1} we know
\[
\nchi_B{\bf A}(\lambda, b, {\bf c}) \geq \Big{(}K\nchi_BB(\lambda, b, {\bf c})+ \frac1N \nchi_B C^2(\lambda, b, {\bf c}) \Big{)}\,\mm.
\]
Combining this observation and the linearity of ${\bf A}, B, C$ with respect to $b$, we can replace the constant $b$ in \eqref{eqn-1} by an arbitrary simple function.  Pick a sequence of simple functions $\{b_n\}_n$ such that $b_n \to \lambda f$ in $L^\infty(M)$. Since ${\bf \Gamma}_2(f,g)$ and $\H_f(\nabla g, \nabla h)\,\mm$ have finite total variation for any $f, g, h \in {\rm TestF}$,  we can see that ${\bf A}(\lambda, b_n, {\bf c})$, $B(\lambda, b_n, {\bf c})\mm$,  $C^2(\lambda, b_n, {\bf c})\mm$  converge to ${\bf A}(\lambda, \lambda f, {\bf c})$, $B(\lambda, \lambda f, {\bf c})\mm$, $C^2(\lambda, \lambda f, {\bf c})\mm$ respectively with respect to the total variation norm $\| \cdot \|_{\rm TV}$. Therefore, we can replace $b$ in \eqref{eqn-1} by $\lambda f$. For the same reason, we can replace  $c_i$ by $h_i$. Then we obtain the following inequality.

\begin{equation}\label{eqn-2}
{\bf A}'(\lambda) \geq \Big{(} KB'(\lambda)+ \frac1N (C')^2(\lambda)\Big{)}\,\mm,
\end{equation}
where
\begin{eqnarray*}
{\bf A}'(\lambda)&=&  4\mathop{\sum}_i(\lambda g+\lambda) \H_f(\nabla h_i, \nabla h_i)\,\mm + 2\lambda^2|\D f|^2 |\D g|^2\,\mm +2\lambda^2 |\la \nabla f, \nabla g \ra|^2\,\mm  \\ &+& 4\mathop{\sum}_{i,j} |\la \nabla h_i, \nabla h_j \ra|^2\,\mm  +
8\lambda \mathop{\sum}_i \la \nabla f, \nabla h_i \ra \la \nabla g, \nabla h_i \ra\,\mm + (\lambda g+\lambda)^2 {\bf \Gamma}_2(f)\\
&+& 4\lambda (\lambda g+\lambda)\H_f(\nabla f,\nabla g)
\\
B'(\lambda)&=&  (\lambda g+\lambda)^2 |\D f|^2\\
C'(\lambda)&=&  (\lambda g+\lambda) \Delta f+ 2\lambda \la \nabla f, \nabla g \ra+2\mathop{\sum}_i |\D h_i|^2.
\end{eqnarray*}

It can be seen that $\H_f(\nabla f,\nabla g)=\frac12 \la \nabla |\D f|^2, \nabla g \ra$. Therefore, all the terms in ${\bf A}', B'\,\mm$ and $C'\,\mm$ vary continuously w.r.t. $\|\cdot \|_{\rm TV}$ as $g$ varies in $W^{1,2}(M)$. Hence the inequality \eqref{eqn-2} holds for any Lipschitz function $g$ with bounded support.
In particular, we can pick $g$ identically 1 on some bounded set $\Omega \subset X$, so that we have $|\D g| = 0$ and $\H_f (\nabla f, \nabla g) = 0$ $\mm$-a.e. on $\Omega$. By the arbitrariness of $\Omega$ we can replace $g$ by ${\bf 1}$ which is the function identically equals to 1 on $X$. Then the inequality \eqref{eqn-2} becomes:

\begin{eqnarray*}
&&\lambda^2 {\bf \Gamma}_2(f)+\big {(} 2\lambda \sum_i \H_f(\nabla h_i, \nabla h_i)+\sum_{i,j} |\la \nabla h_i, \nabla h_j \ra |^2-K\lambda^2  |\D f|^2 \big{)} \mm \\
&-& \Big{(} \lambda^2 \frac {(\Delta f)^2}N+ 2\lambda \frac{\Delta f}{N}|\D h_i|^2+ \frac{(\sum_i  |\D h_i|^2)^2}{N} \Big{)} \mm \geq 0.
\end{eqnarray*}

Let $\gamma_2(f) \,\mm$ be the absolutely continuous  part of ${\bf \Gamma}_2(f)$.  By Lemma \ref{lemma-qf} we have the inequality
\begin{eqnarray*}
&&\left |\sum_i \Big{(}\H_f(\nabla h_i, \nabla h_i)- \frac{\Delta f}{N}|\D h_i|^2 \Big{)}\right |^2 \\&\leq& \Big{(} \gamma_2(f)-K|\D f|^2 -\frac {(\Delta f)^2}N\Big{)}\Big{(}\sum_{i,j}|\la \nabla h_i, \nabla h_j \ra |^2 -\frac{(\sum_i  |\D h_i|^2)^2}{N}\Big{)}.
\end{eqnarray*}

In particular, since $\gamma_2(f)-K|\D f|^2 -\frac {(\Delta f)^2}N \geq 0$ (by Proposition \ref{becondition}), we have
\[
\sum_{i,j}|\langle \nabla h_i, \nabla h_j \rangle |^2  \geq \frac{(\sum_i |\D h_i|^2)^2}{N},~~~~~\mm-\text{a.e.} .
\]
This inequality remains true if  we  replace $\nabla h_i$ by $v:=\sum_k \nchi_{A_k}  \nabla f_k$ where $f_k$ are test functions and $A_k$ are disjoint Borel sets. Therefore by density we can replace $\{\nabla h_i \}_1^{n}$ by any $\{e_{i,m}\}_{i=1}^{n}$ which is a unit  orthogonal subset of $L^2(TM)\restr{E_m}$, whose existence is guaranteed by Proposition \ref{decomposition} and the choice of $m,n, E_m$ at the beginning of the proof. Then we obtain

\[
n=\sum_{i,j}|\langle  e_{i,m},  e_{j,m} \rangle |^2  \geq \frac{(\sum_i |e_{i,m}|^2)^2}{N}=\frac{n^2}N,~~~~~\mm-\text{a.e.}~\text{on}~~E_m,
\]
which implies   $n\leq N$ on $E_m$. Since the finite integer $n  \leq m$ was chosen  arbitrarily, we deduce $m \leq N$. Furthermore, if $E_N$ has positive measure, we obtain
\[
\left |\sum_{i=1}^N \H_f(e_{i, N},e_{i, N})- \sum_{i=1}^N \frac{\Delta f}{N}|e_{i,N}|^2 \right |=0,~~~~~\mm-\text{a.e.~on}~  E_N,
\]
where $\{e_{i,N}\}_{i=1}^{N}$  is a unit  orthogonal basis on $E_N$. This is the same as to say that $\tr \H_f=\Delta f$,  $\mm$-a.e. on $ E$.

\end{proof}

According to this proposition, on  $\rcdkn$ spaces,  we can see that the pointwise Hilbert-Schmidt norm $|T |_{\rm HS}$  of a $L^\infty$-bilinear map $T:[L^2(TM)]^2 \mapsto L^0(M)$ can be defined in the following  way.  We denote $\dim_{\rm loc}:M \mapsto \mathbb{N}$ as the local dimension which is  defined as $\dim_{\rm loc}(x)=n$ on $E_n$, where $\{E_n\}_{n \in \mathbb{N} \cup \{\infty\}}$ is the partition of $X$ in Proposition \ref{decomposition}. Let $T_1, T_2 :[L^2(TM)]^2 \mapsto L^0(M)$ be symmetric bilinear maps, we define $\la T_1, T_2 \ra_{\rm HS}$ as a function such that $\la T_1, T_2 \ra_{\rm HS}:=\sum_{i,j}T_1(e_{i,n}, e_{j,n})T_2(e_{i,n}, e_{j,n})$, $\mm$-a.e. on $E_n$,  where $\{E_n\}_{n\leq N}$ is the partition of $X$ in Proposition \ref{decomposition} and $\{e_{i,n}\}_i, n=1,..., \lfloor N \rfloor$ are the corresponding unit orthogonal  basis. Clearly, this definition is well posed.  In particular, we define the Hilbert-Schmidt norm of $T_1$ by $\sqrt{\la T_1, T_1 \ra_{\rm HS} }$ and denote it by $|T_1|_{\rm HS}$,  and the trace of $T_1$ can be written as $\tr T_1= \la T_1, \Id_{\dim_{\rm loc}} \ra_{\rm HS}$ where  $\Id_{\dim_{\rm loc}}$ is the unique  map satisfying $\Id_{\dim_{\rm loc}}(e_{i,n},
e_{j,n})=\delta_{ij}$, $\mm$-a.e. on $E_n$.

In the following theorem, we prove an improved Bochner inequality.

\begin{theorem}\label{th}
Let $M=\ms$ be a $\rcdkn$ metric measure space. Then
\begin{eqnarray*}
{\bf \Gamma}_2(f) &\geq& \Big{(} K|\D f|^2+|\H_f|^2_{\rm HS}+\frac1{N-\dim_{\rm loc}}(\tr \H_f -\Delta f)^2 \Big{)} \, \mm
\end{eqnarray*}
holds for any $f \in {\rm TestF}$, where $\frac1{N-\dim_{\rm loc}}(\tr \H_f -\Delta f)^2 $ is taken $0$  by definition on the set $\{x: \dim_{\rm loc}(x)=N\}$.
\end{theorem}

\begin{proof}
We define the function $\Phi$ as
\[
\Phi(x, y_1, ... , y_N):=x-\frac12 \sum^N_{i,j} c_{i,j}(y_i-c_i)(y_j-c_j)- c \sum_{i=1}^N (y_i-c_i)^2+C_0
\]
where $c, c_i, c_{i,j}=c_{j,i}$ are constants, $C_0=\frac12 \sum^N_{i,j} c_{i,j}c_ic_j+c\sum_{i=1}^N c_i^2$. Then  we have
\begin{eqnarray*}
\Phi_{x,i}=\Phi_{x,x}&=& 0, ~~~~~~\Phi_{i,j}= -c_{i,j}-2c \delta_{ij}\\
\Phi_x &=& 1,~~~~~~ \Phi_i=-\sum_j c_{i,j}(y_j-c_j)-2c (y_i-c_i).
\end{eqnarray*}

Let $f, h_1, ..., , h_N$ be test functions. Using the chain rules we have

\begin{eqnarray*}
&&|\D \Phi(f, h_1, ... ,h_N)|^2= |\D f|^2+ \sum_i (h_i-c_i) I_i,\\
&&{\bf \Gamma}_2(\Phi(f, h_1, ... ,h_N))= {\bf \Gamma}_2(f)- 2\sum_{i,j}(c_{i,j}+2c \delta_{ij}) \rm H_f(\nabla h_i, \nabla h_j) \,\mm\\
&&-\sum_{i,j,k,l}(c_{i,j}+2c \delta_{ij})(c_{k,l}+2c \delta_{kl}) \la  \nabla h_i, \nabla h_k \ra\la \nabla h_l, \nabla h_j \ra \,\mm+\sum_i (h_i-c_i) {\bf J}_i,\\
&&\Delta \Phi(f, h_1, ... ,h_N)= \Delta f- \sum_{i,j}(c_{i,j}+2c \delta_{ij}) \la \nabla h_i, \nabla h_j \ra+ \sum_i (h_i-c_i) K_i
\end{eqnarray*}
where $\{I_i, K_i\}_i$ are some $L^1(M)$-integrable  terms and $\{{\bf J}_i\}_i$ are measures with finite mass.

Then we apply Lemma \ref{self improve} to the function $\Phi(f, h_1, ... ,h_N)$ to obtain the  inequality

\begin{eqnarray*}
&& {\bf \Gamma}_2(f)- 2\sum_{i,j}(c_{i,j}+2c \delta_{ij}) \H_f(\nabla h_i, \nabla h_j) \,\mm\\
&-&
\sum_{i,j,k,l}(c_{i,j}+2c \delta_{ij})(c_{k,l}+2c \delta_{kl}) \la \nabla h_i, \nabla h_k \ra \la \nabla h_l, \nabla h_j \ra \,\mm +\sum_i (h_i-c_i) J_i \,\mm\\
&\geq & K \Big{(}|\D f|^2+ \sum_i (h_i-c_i) I_i \Big{)}\,\mm +\\
&& \frac1N \Big{(}\Delta(f)- \sum_{i,j}(c_{i,j}+2c \delta_{ij}) \la \nabla h_i, \nabla h_j \ra+ \sum_i (h_i-c_i) K_i
\Big{)}^2\,\mm.
\end{eqnarray*}

Using the same argument as in the proof of last Proposition, we can replace the constants $c_i$ by any simple function. Furthermore, by an approximation argument we can replace the constant $c, c_i, c_{i,j}$ by arbitrary $L^2$ functions. Then pick $c_i=h_i$, the inequality becomes

\begin{eqnarray*}
&& {\bf \Gamma}_2(f)- 2\sum_{i,j}(c_{i,j}+2c \delta_{ij}) \H_f(\nabla h_i, \nabla h_j)\,\mm\\
&-&
\sum_{i,j,k,l}(c_{i,j}+2c \delta_{ij})(c_{k,l}+2c \delta_{kl}) \la \nabla h_i, \nabla h_k \ra \la \nabla h_l, \nabla h_j \ra \,\mm \\
&\geq & K |\D f|^2\,\mm +
 \frac1N \Big{(}\Delta(f)- \sum_{i,j}(c_{i,j}+2c \delta_{ij}) \la \nabla h_i, \nabla h_j \ra
\Big{)}^2\,\mm.
\end{eqnarray*}

Now we restrict the inequality above on Borel set $E_n$, $n \leq N$ where
 $\{E_n\}_{n=1}^{\lfloor N \rfloor}$ is the partition of $X$ in Proposition \ref{decomposition}. The inequality remains true if  we  replace $\nabla h_i$ by $v:=\sum_k \nchi_{A_k}  \nabla f_k$ where $f_k$ are test functions and $A_k$ are disjoint Borel subsets of $E_n$. Therefore due to the density of test functions,  we can replace $\{\nabla h_i \}_1^{n}$ by any $\{e_{i,n}\}_{i=1}^{n}$ which is a unit  orthogonal basis of $L^2(TM)\restr{E_n}$. Doing this replacement on every $E_n$, we obtain
\begin{eqnarray*}
 &&{\bf \Gamma}_2(f)- \Big{(}2\sum_{i,j=1}^{\dim_{\rm loc}}(c_{i,j}+2c \delta_{ij}) \H_f(e_{i,{\dim_{\rm loc}}}, e_{j,{\dim_{\rm loc}}})
+
\sum_{i,j=1}^{\dim_{\rm loc}}(c_{i,j}+2c \delta_{ij})(c_{i,j}+2c \delta_{ij})\Big{)}\,\mm\\
&\geq&  K |\D f|^2\,\mm +
 \frac1N \Big{(}\Delta(f)- \sum_{i=1}^{\dim_{\rm loc}}(c_{i,i}+2c )
\Big{)}^2\,\mm.
\end{eqnarray*}

Pick
\[
c=\left\{\begin{array}{ll}
\displaystyle{\frac{N \tr {\H}_f -\dim_{\rm loc} \Delta f -(N-\dim_{\rm loc}) \tr C}{2n(N-\dim_{\rm loc})}}\qquad &\text{on}~\{x: \dim_{\rm loc}(x) \neq N\} \\
0 \qquad &\text{on}~ \{x: \dim_{\rm loc}(x)=N\}
\end{array}
\right.
\] in the inequality above, where $C=(c_{i,j})$ is a  symmetric $\dim_{\rm loc} \times \dim_{\rm loc}$-matrix whose ${i,j}$ entry is given by the function $c_{i,j}$. Using the local basis vector fields $e_{i,n}$ we can associate to $C$ the $L^\infty(M)$-bilinear map from 
$[L^2(TM)]^2$ with values in $L^0(M)$ sending $\left ( \sum_{i}a_{i,n}e_{i,n}, \sum_j b_{j, n} e_{j,n}\right )$ to $\sum_{i,j}a_{i,n} b_{j, n} c_{i,j}$ and abusing a bit the notation we shall call this map $C$ as well.

Then we obtain the inequality
\begin{eqnarray*}
&~&{\bf \Gamma}_2(f) \geq  \Big{(}K|\D f|^2 - |C|_{\rm HS}^2+ 2\la C, \H_f \ra_{\rm HS}\Big{)} \,\mm\\ &+&\Big{(}\frac{1}{N} (\tr C)^2 +\frac{1}{N} (\Delta f)^2 -\frac{2}{N} (\Delta f)( \tr C)\Big{)}\,\mm\\
&+& \frac { \big{(} (N \tr \H_f -\dim_{\rm loc} \Delta f)^2+(\dim_{\rm loc}-N)^2 (\tr C)^2+2(\dim_{\rm loc}-N)(\tr C)(N \tr \H_f- \dim_{\rm loc} \Delta f) \big{)}}{N\dim_{\rm loc}(N-\dim_{\rm loc})}\,\mm\\
&=& \Big{(}K|\D f|^2 + \frac 1N (\Delta f)^2+\frac 1{N\dim_{\rm loc}(N-\dim_{\rm loc})} (N \tr \H_f -\dim_{\rm loc} \Delta f)^2 \Big{)}\,\mm\\
&-& \Big{(}| C-\H_f +\frac{\tr \H_f }{\dim_{\rm loc}} \Id_{\dim_{\rm loc}} |_{\rm HS}^2+| \H_f-\frac {\tr \H_f}{\dim_{\rm loc}} \Id_{\dim_{\rm loc}}|_{\rm HS}^2+\frac{1}{\dim_{\rm loc}} (\tr C)^2\Big{)}\,\mm
\end{eqnarray*}
where $\la \cdot, \cdot \ra_{\rm HS}$ is the inner product induced by the Hilbert-Schmidt norm and $\Id_{\dim_{\rm loc}}$ is the $\dim_{\rm loc}$-identity  matrix. It can be seen from the Proposition \ref{finite dim} that this inequality still  makes sense if we accept $\frac00=0$.

Then we pick $C=\H_f-\frac{\tr \H_f }{\dim_{\rm loc}} \Id_{\dim_{\rm loc}}$ in this inequality and finally obtain
\begin{eqnarray*}
{\bf \Gamma}_2(f) &\geq &  \Big{(}K|\D f|^2 + \frac 1N (\Delta f)^2+\frac 1{N{\dim_{\rm loc}}(N-{\dim_{\rm loc}})} (N \tr \H_f -{\dim_{\rm loc}} \Delta f)^2\Big{)}\,\mm \\
&+& | \H_f-\frac {\tr \H_f}{\dim_{\rm loc}} \Id_{\dim_{\rm loc}}|_{\rm HS}^2\,\mm\\
&=& \Big{(}K|\D f|^2 +| \H_f|_{\rm HS}^2+\frac 1{(N-{\dim_{\rm loc}})} (\tr \H_f -\Delta f)^2\Big{)}\,\mm,
\end{eqnarray*}
which is the thesis.

\end{proof}


\section{$N$-Ricci tensor}

In this section, we use the improved version of Bochner inequality  obtained in the last section to give a definition of $N$-Ricci tensor.

We recall that the class of test vector fields ${\rm TestV}(M) \subset L^2(TM)$  is defined  as
\[
{\rm TestV}(M):= \{\mathop{\sum}_{i=1}^n  g_i   \nabla  f_i: n\in \mathbb{N}, f_i, g_i \in  {\rm TestF}(M), i= 1, ... , n\}.
\]
It can be proved that ${\rm TestV}(M)$ is dense in $L^2(TM)$ when $M$ is ${\rm RCD }$.

Let $X=\sum_ig_i \nabla f_i \in {\rm TestV}(M)$ be a test vector field.  We  define $\nabla X \in L^2(TM)\otimes L^2(TM)$ by the following formula:
\[
\la \nabla X, v_1 \otimes v_2 \ra_{L^2(TM)\otimes L^2(TM)}:= \sum_i\la \nabla g_i, v_1 \ra \la \nabla f_i, v_2 \ra+ \sum_ig_i\H_{f_i}(v_1, v_2),~~~~\forall v_1, v_2 \in {\rm TestV}(M).
\]
It can be seen that this definition is well posed and that the completion of ${\rm TestV}(M)$ with respect to the norm $\| \cdot \|_{C}:=\sqrt {\| \cdot \|_{L^2(TM)}^2+ \int | \nabla \cdot |^2_{L^2(TM)\otimes L^2(TM)}\,\mm}$ can be identified with a subspace of
 $L^2(TM)$, which is denoted by $H^{1,2}_C(TM)$.

Let $X \in L^2(TM)$. We say that $X \in {\rm D}({\rm div})$ if there exists a  function $g\in L^2(M)$ such that
\[
\int h g\,\mm=-\int \la \nabla h, X\ra\,\mm
\]
for any $h \in W^{1,2}(M)$. We then denote  such  function which is clearly unique by ${\rm div}X$.

It is easy to see that ${\rm div}\cdot $ is a linear operator on ${\rm D}({\rm div})$, that ${\rm TestV}(M) \subset {\rm D}({\rm div})$ and  that the formula:
\[
{\rm div}(g\nabla f)=\la \nabla g, \nabla f\ra +g \Delta f,~~~~f,g \in {\rm TestF}(M)
\]
holds.

It is unknown whether there is any inclusion  relation between ${\rm D}({\rm div})$ and $H^{1,2}_C(TM)$. However, in  \cite{G-N} it has been introduced  the space $(H^{1,2}_H(TM), \| \cdot \|_{H^{1,2}_H} )$ which is contained in both ${\rm D}({\rm div})$ and $H^{1,2}_C(TM)$, and will be useful for our purposes. In the smooth setting, $H^{1,2}_H(TM)$ would be the space of vector fields corresponding to $L^2$ 1-forms having both exterior derivative and co-differential in $L^2(M)$.

The properties of $H^{1,2}_H(TM)$ that we shall need are:
\begin{itemize}
\item [(a)] ${\rm TestV}(M)$ is dense in $H^{1,2}_H(TM)$,
\item [(b)] $H^{1,2}_H(TM)$ is contained in $H^{1,2}_C(TM)$ with continuous embedding,
\item  [(c)] $ H^{1,2}_H(TM) \subset {\rm D}({\rm div})$ and for any $X_n \to X$  in $H^{1,2}_H(TM)$, we have ${\rm div} X_n \to {\rm div}X$ in $L^2(M)$.
\end{itemize}

Now, we can generalize the Proposition \ref{finite dim}  in the following way. We denote the natural correspondences (dualities) between  $L^2(TM)$   and $ L^2(T^*M)$ by $(\cdot )^b$ and $(\cdot )^\sharp$ (same notation for $L^2(TM)\otimes L^2(TM)$   and $ L^2(T^*M)\otimes L^2(T^*M)$). For example,
$(\nabla f)^b=\d f$, $(\H_f)^\sharp =\nabla \nabla f$ for $f\in {\rm TestF}(M)$. In this case, we know $\la T_1, T_2 \ra_{\rm HS}=\la (T_1)^\sharp, (T_2)^\sharp \ra_{L^2(TM)\otimes L^2(TM)}$ for any $T_1, T_2 \in  L^2(T^*M)\otimes L^2(T^*M)$.

\begin{proposition}
Let $M=\ms$ be a $\rcdkn$ metric measure space, $E\subset X$ be a Borel set. Assume that the local dimension of $M$ on $E$ is $N$, then
$\tr (\nabla X)^b={\rm div}X$ $\mm$-a.e. $x \in E$ for any $X \in H_H^{1,2}(TM)$.
\end{proposition}

\begin{proof}
Thanks to the Proposition \ref{finite dim}, it is sufficient to prove the equality
\begin{equation}\label{eq-1}
\tr (\nabla X)^b={\rm div}X~~~ \mm-\text{a.e. on}~~ E
\end{equation}
for any $X \in H^{1,2}_H(TM)$, under the assumption  that  $\tr \H_f =\Delta f$ $\mm$-a.e. on  $E$ for any $f \in {\rm TestF}(M)$.

First of all, as we know $(\nabla \nabla f)^b=\H_f$ and ${\rm div}(\nabla f)=\Delta f$, the equality \eqref{eq-1} holds for every  $X$ of the form $\nabla f$  for some $f\in {\rm TestF}(M)$.

Secondly, for any $X=\sum_i g_i \nabla f_i \in {\rm TestV}$, the assertion holds by recalling the identities
 $\nabla (g \nabla f)= \nabla g \otimes \nabla f+ g \nabla(\nabla f)$ and ${\rm div}(g \nabla f)= \la \nabla g, \nabla f\ra+g {\rm div}(\nabla f)$.

 Finally, for any $X \in H_H^{1,2}(TM)$, we can find a sequence $\{X_i\}_i \subset {\rm TestV}$ such that $X_i \to  X$ in $H_H^{1,2}(TM)$. Therefore $\la (\nabla X_i)^b, \Id_{\dim_{\rm loc}} \ra_{\rm HS} \to \la (\nabla X)^b, \Id_{\dim_{\rm loc}} \ra_{\rm HS}$ in $L^2$ because $\|\cdot \|_{ H^{1,2}_H}$ convergence is stronger than the $\|\cdot \|_{ H^{1,2}_C}$ convergence. Then $\tr (\nabla X_{i})^b=\la (\nabla X_i)^b, \Id_{\dim_{\rm loc}} \ra_{\rm HS} \to \la (\nabla X)^b, \Id_{\dim_{\rm loc}} \ra_{\rm HS}= \tr (\nabla X)^b$ in $L^2$.  Since $X_{i} \to X$ in $H^{1,2}_H(TM)$ implies ${\rm div}X_i \to {\rm div}X$ in $L^2$, we conclude that  ${\rm div}X= \tr (\nabla X)^b$ $\mm$-a.e.  on $ E$.
\end{proof}


We shall now use the result of Theorem \ref{th} to define the $N$-Ricci tensor.

We start  defining ${\bf \Gamma}_2(\cdot, \cdot): [{\rm TestV}(M)]^2 \mapsto {\rm Meas}(M)$ by
\begin{eqnarray*}
{\bf \Gamma}_2(X, Y):={\bf \Delta} \frac{\la X, Y \ra}{2}+\Big{(} \frac12 \la X, (\Delta_{\rm H} Y^b)^\sharp \ra +\frac12 \la Y, (\Delta_{\rm H} X^b)^\sharp \ra \Big{)} \, \mm,
\end{eqnarray*}
where $(X, Y) \in [{\rm TestV}(M)]^2$ and $\Delta_{\rm H}$ is the Hodge Laplacian. It is proved in \cite{G-N} that ${\bf \Gamma}_2(\nabla f, \nabla f)={\bf \Gamma}_2(f)$ for $f \in {\rm TestF}(M)$ and that ${\bf \Gamma}_2(\cdot, \cdot)$ can be continuously  extended to $[H^{1,2}_H(TM)]^2$. Furthermore, it is known from Theorem 3.6.7 of \cite{G-N}  that $\bRic(X, Y):={\bf \Gamma}_2(X, Y)- \la (\nabla X)^b, (\nabla Y)^b \ra_{\rm HS} \, \mm$ is a symmetric ${\rm TestF}(M)$-bilinear form on $[H^{1,2}_H(TM)]^2$.

We then define the map ${R}_N$ on  $[H^{1,2}_H(TM)]^2$  by
\[
{ R}_N(X,Y):=\left\{\begin{array}{lll}
\displaystyle{\frac1{N-{\dim_{\rm loc}}}\big{(}\tr (\nabla X)^b -{\rm div} X\big{)}\big{(}\tr (\nabla Y)^b -{\rm div} Y\big{)} }  &{\dim_{\rm loc}}<N,\\
0 &{\dim_{\rm loc}}\geq N.
\end{array}
\right.
\]
From the continuity of ${\rm div}\cdot$ and $\tr (\nabla \cdot )^b$ on
$H^{1,2}_H(TM)$,  we deduce that $(X, Y) \mapsto { R}_N(X,Y)\,\mm$ is continuous on  $[H^{1,2}_H(TM)]^2$ with values in ${\rm Meas}(M)$.
From the calculus rules developed in \cite{G-N}, it is easy to see  that $(X, Y) \mapsto {R}_N(X, Y)\,\mm$ is homogenous with respect to the multiplication of test functions, i.e.
\[
f {R}_N(X, Y)\,\mm={ R}_N(f X, Y)\,\mm
\]
for any $f  \in {\rm TestF}(M)$. Therefore we can define $\bRicn(\cdot ,\cdot)$ on $[H^{1,2}_H(TM)]^2$ in the following way:

\begin{definition}[Ricci tensor] \label{Ricci tensor}
We define $\bRicn$ as a measure valued  map on $[H^{1,2}_H(TM)]^2$ such that for any $X, Y \in H^{1,2}_H(TM)$ it holds
\begin{eqnarray*}
\bRicn(X,Y)&=&
{\bf \Gamma}_2(X, Y) -\la (\nabla X)^b, (\nabla Y)^b \ra_{\rm HS} \, \mm-{R}_N(X, Y)\,\mm.
\end{eqnarray*}
\end{definition}

Combining the discussions above and Proposition \ref{finite dim}, we know $\bRicn$ is a well defined tensor, i.e. $(X, Y) \mapsto \bRicn(X, Y)$ is a symmetric ${\rm TestF}(M)$-bilinear form. Then, we can prove the following theorem by combining our  Theorem \ref{th} and Theorem 3.6.7 of \cite{G-N}.

\begin{theorem} \label{th-ricci}
Let $M$ be a $\rcdkn$ space. Then
\[
\bRicn(X, X) \geq K|X|^2\,\mm,
\]
 and
\begin{eqnarray}\label{ineq-main}
{\bf \Gamma}_2(X,X) &\geq & \Big{(} \frac{({\rm div}X)^2}{N}+\bRicn(X, X) \Big{)}\,\mm
\end{eqnarray}
holds for any $X \in H^{1,2}_H(TM)$.
Conversely, on a ${\rm RCD}(K', \infty)$ space $M$, assume that
\begin{itemize}
\item [(1)] ${\rm dim}_{\rm max} M \leq N$
\item [(2)] $\tr (\nabla X)^b={\rm div}X~ \mm-\text{a.e. on}~~ \{{\dim}_{\rm loc}=N\}, \forall X \in H^{1,2}_H(TM)$
\item [(3)]  $\bRicn \geq K$
\end{itemize}
for some $K \in \R$, $N \in [1,+\infty]$, then it is  $\rcdkn$.
\end{theorem}
\begin{proof}
Assume that we have the decomposition:
\[
\bRicn(X, X)=\bRicn^{ac}(X, X)+\bRicn^{sing}(X, X)
\]
where $\bRicn^{ac}(X, X)={\rm Ricci}_N(X, X)\,\mm$ is the absolutely continuous part of $\bRicn(X, X)$ and $\bRicn^{sing}(X, X)$ is its  singular part. Then we need to prove that
\begin{equation}\label{thineq-1}
{\rm Ricci}_N(X, X)\,\mm  \geq K|X|^2\,\mm
\end{equation}
and
\begin{equation}\label{thineq-2}
\bRicn^{sing}(X, X)  \geq 0
\end{equation}
for any $X \in H^{1,2}_H(TM)$.

From the definition of $\bRicn$ and $\bRic$ we know that $\bRicn^{sing}(X, X)$ coincides with the singular part of $\bRic(X, X)$. It is proved in Lemma 3.6.2, \cite{G-N} that $\bRic(X, X) \geq K|X|^2\,\mm$. Therefore \eqref{thineq-2} holds for any $X \in H^{1,2}_H(TM)$.

We turn to prove \eqref{thineq-1} for any $X \in H^{1,2}_H(TM)$. It can be seen from the definition  that \eqref{thineq-1} means
\begin{eqnarray}\label{ineq-1}
\gamma_2(X, X)\,\mm \geq  \Big{(}K|X|^2+|(\nabla X)^b|^2_{\rm HS}
+\frac1{N-{\dim_{\rm loc}}}(\tr (\nabla X)^b -{\rm div} X)^2\Big{)}\, \mm
\end{eqnarray}
where $\gamma_2(X, X)\,\mm$ is the absolutely continuous part of ${\bf \Gamma}_2(X, X)$.

First of all, notice that for $X=\nabla f$, \eqref{ineq-1} is exactly the inequality in Theorem \ref{th}. Hence \eqref{thineq-1} holds for any $X=\nabla f$, $f \in {\rm TestF}(M)$.

Secondly, we need to prove \eqref{thineq-1} for any $X \in {\rm TestV}(M)$.  Let $X=\sum_i g_i \nabla f_i$ be a test vector field.
From the homogeneity of ${R}_N\,\mm$ and $\bRic:={\bf \Gamma}_2(\cdot, \cdot) -\la (\nabla \cdot)^b, (\nabla \cdot)^b \ra_{\rm HS} \, \mm$ which is proved in \cite{G-N} we know that $\bRicn(\cdot, \cdot)$ is a symmetric ${\rm TestF}(M)$-bilinear form. In particular, 
${\rm Ricci}_N(\cdot, \cdot)\,\mm$ is a symmetric ${\rm TestF}(M)$-bilinear form. Therefore ${\rm Ricci}_N(X, X) =\sum_{i,j} g_ig_j{\rm Ricci}_N(\nabla f_i, \nabla f_j)$ $\mm$-a.e.. Thus we need  to prove the inequality
\begin{equation}\label{ineq-2}
\sum_{i,j} g_ig_j{\rm Ricci}_N(\nabla f_i, \nabla f_j)\,\mm \geq K\sum_{i,j}g_ig_j\la \nabla f_i, \nabla f_j \ra\,\mm.
\end{equation}
 Hence by  approximation with simple functions (see Lemma \ref{lemma-approx} below), it is sufficient to prove this inequality for simple functions $g_i=\sum_{k_i=1}^{K_i} a_{i,k_i}\nchi_{E_{i,k_i}}$, i.e.
 \begin{equation*}
\sum_{i,j,k_i,k_j} a_{i,k_i}a_{j,k_j} \nchi_{E_{i,k_i} \cap E_{j,k_j}}{\rm Ricci}_N(\nabla f_i, \nabla f_j)\,\mm \geq K\sum_{i,j,k_i,k_j} a_{i,k_i}a_{j,k_j} \nchi_{E_{i,k_i} \cap E_{j,k_j}}\la \nabla f_i, \nabla f_j \ra\,\mm.
\end{equation*}

Let $E \in X$ be a Borel set with positive measure such that $E=\cap_I (E_{i,k_i} \cap E_{j,k_j})$ where $I:=\{(i,j,k_i,k_j):\mm(E\cap E_{i,k_i} \cap E_{j,k_j})>0$. We then restrict the inequality above on $E$
\[
\mathop{\sum}_{ (i,j,k_i,k_j) \in I}{\rm Ricci}_N(\nabla a_{i,k_i}f_i, \nabla a_{j,k_j}f_j) \,\mm\restr{E} \geq K \mathop{\sum}_{(i,j,k_i,k_j)\in I}\la \nabla a_{i,k_i}f_i, \nabla a_{j,k_j}f_j \ra\,\mm\restr{E},
\]
which is equivalent to
\[
\bRicn(\nabla F, \nabla F) \restr{E} \geq K |\D F|^2\,\mm\restr{E},
\]
where $F=\mathop{\sum}_{(i,k_i): \exists (i,j,k_i,k_j) \in I}a_{i,k_i}f_i$. Clearly, this is true due to Theorem \ref{th}. Then we can repeat this argument on all $E$ which is a decomposition of $X$ and prove the assertion.

Next, it is sufficient to prove that \eqref{ineq-1} can be continuously extended to $ H^{1,2}_H(TM)$.  It is proved in Theorem 3.6.7 of \cite{G-N} that ${\bf \Gamma}_2(X,X)-|(\nabla X)^b|^2_{\rm HS} \, \mm $ vary continuously w.r.t. $\|\cdot \|_{\rm TV}$ as $X$ varies in $H_H^{1,2}(M)$. The term $\frac1{N-{\dim_{\rm loc}}}(\tr (\nabla X)^b -{\rm div} X)^2\,\mm$ also varies continuously in ${\rm Meas}(M)$ due to the property (b) and (c) of $H^{1,2}_H(TM)$. Therefore we know \eqref{ineq-1} holds for all $X \in H^{1,2}_H(TM)$.

Moreover, from the definition of $\bRicn$ we can see that
\begin{eqnarray*}
\Big{(} \frac{({\rm div}X)^2}{N}+\bRicn(X, X) \Big{)}\,\mm
&=& {\bf \Gamma}_2(X, X)-|(\nabla X)^b|_{\rm HS}^2\,\mm+\frac{({\rm div}X)^2}{N}\,\mm\\
&-&\frac1{N-{\dim_{\rm loc}}}\big{(}\tr (\nabla X)^b -{\rm div} X\big{)}^2\, \mm\\
&\leq& {\bf \Gamma}_2(X, X)-\frac{(\tr (\nabla X)^b)^2}{{\dim_{\rm loc}}}\,\mm+\frac{({\rm div}X)^2}{N}\,\mm\\
&-&\frac1{N-{\dim_{\rm loc}}}(\tr (\nabla X)^b -{\rm div} X)^2\Big{)}\, \mm\\
&\leq & {\bf \Gamma}_2(X, X)
\end{eqnarray*}
which is the inequality \eqref{ineq-main}.

Conversely, picking $X=\nabla f$, $f \in {\rm TestF}$ in $\bRicn(X, X) \geq K|X|^2\,\mm$, we have the following inequality according to the definition
\begin{eqnarray*}
{\bf \Gamma}_2(f) &\geq& \big{(} K|\D f|^2+|\H_f|^2_{\rm HS}+\frac1{N-{\dim_{\rm loc}}}(\tr \H_f -\Delta f)^2 \big{)} \, \mm.
\end{eqnarray*}

Then by Cauchy-Schwarz inequality we obtain
\begin{eqnarray*}
{\bf \Gamma}_2(f) &\geq& \big{(} K|\D f|^2+\frac1{{\dim_{\rm loc}}} (\tr \H_f )^2 +\frac1{N-{\dim_{\rm loc}}}(\tr \H_f -\Delta f)^2 \big{)} \, \mm \\
&\geq& \big{(} K|\D f|^2+\frac1N (\Delta f)^2\big{)} \, \mm
\end{eqnarray*}
for any $f \in {\rm TestF}(M)$. The conclusion follows Proposition \ref{becondition}.
\end{proof}

\begin{lemma}\label{lemma-approx}
The inequality \eqref{ineq-2} holds for $g_i\in  L^\infty$ if  it holds for $\{g_i\}$ which are simple functions.
\end{lemma}
\begin{proof}
The inequality \eqref{ineq-2} holds if and only only 
\begin{equation}\label{ineq-3}
\int h\sum_{i,j} g_ig_j{\rm Ricci}_N(\nabla f_i, \nabla f_j)\,\d\mm \geq K\int h\sum_{i,j}g_ig_j\la \nabla f_i, \nabla f_j \ra\,\d\mm.
\end{equation}
for any $h\in L^\infty$, $h\geq 0$. Since $g_i$ is $L^\infty$, we can find simple functions $\{g_i^n\}_n$ such that $g_i^n \to g_i$ in ${L^\infty}$.  From hypothesis we know
\begin{equation}\label{ineq-4}
\int h\sum_{i,j} g^n_ig^n_j{\rm Ricci}_N(\nabla f_i, \nabla f_j)\,\d\mm \geq K\int h\sum_{i,j}g^n_ig^n_j\la \nabla f_i, \nabla f_j \ra\,\d\mm.
\end{equation}
From the property of $\bRic$ in Theorem 3.6.7 in \cite{G-N} and the definition of $\bRicn$ we know ${\rm Ricci}_N(\nabla f_i, \nabla f_j)$ are $L^1$.  Letting $n \to \infty$, the both sides in \eqref{ineq-4} converge in $L^1$ and we obtain \eqref{ineq-3}.
\end{proof}

\def\cprime{$'$} \def\cprime{$'$}

\end{document}